\documentclass[reqno,12pt]{amsart}
\usepackage{amscd,amsfonts,amssymb}
\numberwithin{equation}{section}
\newtheorem{Theorem}{Theorem}[section]
\newtheorem{Lemma}{Lemma}[section]
\newtheorem{Corollary}{Corollary}[section]

\theoremstyle{definition}

\theoremstyle{remark}
\newtheorem{Remark}{Remark}[section]
\newtheorem{Example}{Example}[section]

\newcommand{\essinf}{\mathop{\rm ess \, inf}\limits}

\newcommand{\mes}{\mathop{\rm mes}\nolimits}

\author{A.A. Kon'kov}
\address{Department of Differential Equations,
Faculty of Mechanics and Mathematics,
Mo\-s\-cow Lo\-mo\-no\-sov State University,
Vorobyovy Gory,
Moscow, 119992 Russia}
\email{konkov@mech.math.msu.su}
\author{A.E. Shishkov}
\address{
Center of Nonlinear Problems of Mathematical Physics,
RUDN University,
Miklukho-Maklaya str. 6,
Moscow, 117198 Russia;
Institute of Applied Mathematics and Mechanics of NAS of Ukraine,
Dobrovol'skogo str. 1, Slavyansk, 84116 Ukraine
}
\email{aeshkv@yahoo.com}
\thanks{The work of the second author is supported by RUDN University, Project 5-100}
\title[Generalization of the Keller-Osserman theorem]{Generalization of the Keller-Osserman theorem for higher order differential inequalities}
\date{}

\begin{document}

\begin{abstract}
We obtain exact conditions guaranteeing that any global weak solution of the differential inequality 
$$
	\sum_{|\alpha| = m}
	\partial^\alpha
	a_\alpha (x, u)
	\ge
	g (|u|)
	\quad
	\mbox{in } {\mathbb R}^n
$$
is trivial, where $m, n \ge 1$ are integers and $a_\alpha$ and $g$ are some functions.

These conditions generalize the well-know Keller-Osserman condition.
\end{abstract}

\maketitle

\section{Introduction}

We study the differential inequality
\begin{equation}
	\sum_{|\alpha| = m}
	\partial^\alpha
	a_\alpha (x, u)
	\ge
	g (|u|)
	\quad
	\mbox{in } {\mathbb R}^n,
	\label{1.1}
\end{equation}
where $m, n \ge 1$ are integers and $a_\alpha$ are Caratheodory functions such that
$$
	|a_\alpha (x, \zeta)| 
	\le 
	A |\zeta|,
	\quad
	|\alpha| = m,
$$
with some constant $A > 0$ for almost all $x = {(x_1, \ldots, x_n)} \in {\mathbb R}^n$ 
and for all $\zeta \in {\mathbb R}$.
By $\alpha = {(\alpha_1, \ldots, \alpha_n)}$ we mean a multi-index with
$|\alpha| = \alpha_1 + \ldots + \alpha_n$ 
and
$
	\partial^\alpha 
	= 
	{\partial^{|\alpha|} / (\partial_{x_1}^{\alpha_1} \ldots \partial_{x_n}^{\alpha_n})}.
$
It is also assumed that $g$ is a non-decreasing convex function on the interval $[0, \infty)$
and, moreover, $g (\zeta) > 0$ for all $\zeta > 0$.

Let us denote by $B_r^x$ the open ball in ${\mathbb R}^n$ 
of radius $r > 0$ and center at $x$. In the case of $x = 0$, we write $B_r$ instead of $B_r^0$.

A function $u  \in {L_{1, loc} ({\mathbb R}^n)}$ is called a global weak solution of~\eqref{1.1} 
if ${g (|u|)} \in {L_{1, loc} ({\mathbb R}^n)}$ and
\begin{equation}
	\int_{{\mathbb R}^n}
	\sum_{|\alpha| = m}
	(-1)^m
	a_\alpha (x, u)
	\partial^\alpha
	\varphi
	\,
	dx
	\ge
	\int_{{\mathbb R}^n}
	g (|u|)
	\varphi
	\,
	dx
	\label{1.2}
\end{equation}
for any non-negative function $\varphi \in C_0^\infty ({\mathbb R}^n)$.

In their pioneering works~\cite{Keller, Osserman}, J.B.~Keller and R.~Osserman 
proved
that, under the condition
\begin{equation}
	\int_1^\infty
	\left(
		\int_1^\zeta
		g (\xi)
		\,
		d\xi
	\right)^{-1/2}
	\,
	d\zeta
	<
	\infty,
	\label{1.3}
\end{equation}
the elliptic inequality
\begin{equation}
	\Delta u \ge g (u)
	\quad
	\mbox {in } 
	{\mathbb R}^n
	\label{1.4}
\end{equation}
has no positive global solutions.
Since then, a lot of papers appeared on the absence of solutions for various differential equations and inequalities. 
In so doing,
for the general nonlinearity, 
most of these papers 
dealt with second order differential operators~\cite{FPR2008}--\cite{Veron}.
In the case of higher order operators, almost all studies were limited to the Emden-Fowler nonlinearity $g (t) = t^\lambda$~\cite{KS}--\cite{MPbook}.

In our paper, we obtain exact conditions guaranteeing that any global weak solution of inequality~\eqref{1.1} is trivial or, in other words, is equal to zero almost everywhere in ${\mathbb R}^n$.
For inequalities of the form~\eqref{1.4}, these conditions are equivalent to the Keller-Osserman condition~\eqref{1.3}.

\section{Main results}

\begin{Theorem}\label{T2.1}
Let
\begin{equation}
	\int_1^\infty
	g^{- 1 / m} (\zeta)
	\zeta^{1 / m - 1}
	\,
	d\zeta
	<
	\infty
	\label{T2.1.1}
\end{equation}
and
\begin{equation}
	\liminf_{t \to +0} G^{n - m} (t) t < \infty,
	\label{T2.2.1}
\end{equation}
where 
\begin{equation}
	G (t)
	=
	\int_t^\infty
	g^{- 1 / m} (\zeta)
	\zeta^{1 / m - 1}
	\,
	d\zeta.
	\label{T2.1.3}
\end{equation}
Then any global weak solution of~\eqref{1.1} is trivial.
\end{Theorem}

\begin{Theorem}\label{T2.3}
Let
\begin{equation}
	\int_0^\infty
	g^{- 1 / m} (\zeta)
	\zeta^{1 / m - 1}
	\,
	d\zeta
	<
	\infty.
	\label{T2.3.1}
\end{equation}
Then~\eqref{1.1} has no global weak solutions.
\end{Theorem}

\begin{Theorem}\label{T2.4}
Let~\eqref{T2.1.1} be valid, then any global weak solution of~\eqref{1.1} satisfies the estimate
\begin{equation}
	\frac{1}{r^n}
	\int_{B_r}
	|u|
	\,
	dx
	\le
	C G^{-1} (k r)
	\label{T2.4.1}
\end{equation}
for all $r > 0$, where $G^{-1}$ is the inverse function to~\eqref{T2.1.3} 
and the constants $C > 0$ and $k > 0$ depend only on $A$, $m$, and $n$.
\end{Theorem}

The proof of Theorems~\ref{T2.1}--\ref{T2.4} is given in Section~\ref{proof}.

\begin{Remark}\label{R2.1}
If~\eqref{T2.1.1} holds and inequality~\eqref{1.1} has a global weak solution, 
then in accordance with Theorem~\ref{T2.3} we obviously have
\begin{equation}
	\lim_{t \to +0} G (t) = \infty.
	\label{R2.1.1}
\end{equation}
Thus, the right-hand side of~\eqref{T2.4.1} is defined for all $r > 0$.
Since $g$ is a a non-decreasing convex function on $[0, \infty)$, 
condition~\eqref{T2.3.1} implies that $g (0) > 0$.
\end{Remark}

\begin{Theorem}\label{T2.5}
Let~\eqref{T2.1.1} be valid, then
\begin{equation}
	\lim_{r \to \infty}
	\frac{1}{r^n}
	\int_{B_r}
	|u|
	\,
	dx
	=
	0
	\label{T2.5.1}
\end{equation}
for any global weak solution of inequality~\eqref{1.1}.
\end{Theorem}

\begin{Theorem}\label{T2.6}
Let~\eqref{T2.1.1} be valid and, moreover, $m \ge n$. 
Then any global weak solution of~\eqref{1.1} is trivial.
\end{Theorem}

\begin{proof}[Proof of Theorems~$\ref{T2.5}$ and~$\ref{T2.6}$]
Since $G^{-1} (r) \to 0$ as $r \to \infty$, relation~\eqref{T2.5.1} readily follows from estimate~\eqref{T2.4.1} of Theorem~\ref{T2.4}.
In turn, to prove Theorem~\ref{T2.6}, it is sufficient to use Theorem~\ref{T2.1}. 
\end{proof}

\begin{Remark}\label{R2.2}
In the case of $m = 2$, condition~\eqref{T2.1.1} takes the form
\begin{equation}
	\int_1^\infty
	(g (\zeta) \zeta)^{- 1 / 2}
	\,
	d\zeta
	<
	\infty.
	\label{R2.2.1}
\end{equation}
It does not present any particular problem to verify that~\eqref{R2.2.1} is equivalent to 
the well-known Keller-Osserman condition~\eqref{1.3}. 
Really, taking into account the fact that $g$ is a non-decreasing positive function on 
the interval $(0, \infty)$, we obtain
$$
	\int_1^\zeta
	g (\xi)
	\,
	d\xi
	\ge
	\int_{\zeta / 2}^\zeta
	g (\xi)
	\,
	d\xi
	\ge
	\frac{\zeta}{2}
	g 
	\left(
		\frac{\zeta}{2}
	\right),
	\quad
	\zeta \ge 2.
$$
Hence,~\eqref{R2.2.1} implies~\eqref{1.3}.
At the same time,
$$
	\int_1^\zeta
	g (\xi)
	\,
	d\xi
	\le
	\zeta g (\zeta),
	\quad
	\zeta \ge 1;
$$
therefore,~\eqref{R2.2.1} follows from~\eqref{1.3}.
\end{Remark}

\begin{Corollary}[Keller-Osserman]\label{C2.1}
Suppose that~\eqref{R2.2.1} is valid, then any non-negative global weak solution of~\eqref{1.4} 
is trivial.
\end{Corollary}

\begin{proof}
Let $u$ be a non-negative global weak solution of~\eqref{1.4}.
By the submean-value property, we have
$$
	u (x) 
	\le 
	\frac{
		1
	}{
		|B_r|
	}
	\int_{B_r^x}
	u
	\,
	dy
$$
for all $r > 0$ and for almost all $x \in {\mathbb R}^n$, 
where $|B_r|$ is the $n$-dimensional volume of the ball $B_r$,
whence in accordance with Theorem~\ref{T2.5} it follows that $u = 0$ almost everywhere 
in ${\mathbb R}^n$.
\end{proof}

\begin{Example}\label{E2.1}
Consider the inequality 
\begin{equation}
	\sum_{|\alpha| = m}
	\partial^\alpha
	a_\alpha (x, u)
	\ge
	c_0 |u|^\lambda
	\quad
	\mbox{in } {\mathbb R}^n,
	\quad
	c_0 = const > 0.
	\label{E2.1.1}
\end{equation}
By Theorem~\ref{T2.1}, the conditions 
$$
	\lambda > 1
	\quad
	\mbox{and} 
	\quad
	\lambda (n - m) \le n
$$
imply that any global weak solution of~\eqref{E2.1.1} is trivial.
It can be shown that these conditions are the best possible~\cite{KS, MPbook}.
\end{Example}

\begin{Example}\label{E2.2}
Let us examine the critical exponent $\lambda = 1$ in the right-hand side of~\eqref{E2.1.1}.
Namely, consider the inequality
\begin{equation}
	\sum_{|\alpha| = m}
	\partial^\alpha
	a_\alpha (x, u)
	\ge
	c_0 |u| \ln^\nu (2 + |u|)
	\quad
	\mbox{in } {\mathbb R}^n.
	\quad
	c_0 = const > 0.
	\label{E2.2.1}
\end{equation}
By Theorem~\ref{T2.1}, if
\begin{equation}
	\nu > m,
	\label{E2.2.2}
\end{equation}
then any global weak solution of~\eqref{E2.2.1} is trivial.
At the same time, for all positive even integers $m$ and real numbers $\nu \le m$ and $c_0 > 0$ 
the inequality
$$
	\Delta^{m/2} 
	u 
	\ge 
	c_0 |u| \ln^\nu (2 + |u|)
	\quad
	\mbox{in } {\mathbb R}^n
$$
has a positive infinitely smooth global solution.
As such a solution, one can take
$$
	u (x) 
	= 
	e^{
		e^{
			k (1 + |x|^2)^{1 / 2}
		}
	},
$$
where $k > 0$ is a sufficiently large real number.
Thus, condition~\eqref{E2.2.2} is also the best possible.
\end{Example}

\section{Proof of Theorems~\ref{T2.1}--\ref{T2.4}}\label{proof}

In this section, by $C$ we denote various positive constants 
that can depend only on $A$, $m$, and $n$. 

\begin{Lemma}\label{L3.1}
Let $u$ be a global weak solution of~\eqref{1.1}, then
$$
	\int_{
		B_{r_2}
		\setminus
		B_{r_1}
	}
	|u|
	\,
	dx
	\ge
	C
	(r_2 - r_1)^m
	\int_{
		B_{r_1}
	}
	g (|u|)
	\,
	dx
$$
for all real numbers $0 < r_1 < r_2$ such that $r_2 \le 2 r_1$. 
\end{Lemma}

\begin{proof}
Take a non-negative function $\varphi_0 \in C^\infty ({\mathbb R})$ satisfying the conditions
$$
	\left.
		\varphi_0
	\right|_{
		(- \infty, 0]
	}
	=
	0
	\quad
	\mbox{and}
	\quad
	\left.
		\varphi_0
	\right|_{
		[1, \infty)
	}
	=
	1.
$$
Putting
$$
	\varphi (x)
	=
	\varphi_0
	\left(
		\frac{r_2 - |x|}{r_2 - r_1}
	\right)
$$
as a test function in~\eqref{1.2}, we obtain
$$
	\int_{
		{\mathbb R}^n
	}
	\sum_{|\alpha| = m}
	(- 1)^m
	a_\alpha (x, u)
	\partial^\alpha
	\varphi_0
	\left(
		\frac{r_2 - |x|}{r_2 - r_1}
	\right)
	dx
	\ge
	\int_{
		{\mathbb R}^n
	}
	g (|u|)
	\varphi_0
	\left(
		\frac{r_2 - |x|}{r_2 - r_1}
	\right)
	dx.
$$
Combining this with the estimates
$$
	\left|
		\int_{
			{\mathbb R}^n
		}
		\sum_{|\alpha| = m}
		(- 1)^m
		a_\alpha (x, u)
		\partial^\alpha
		\varphi_0
		\left(
			\frac{r_2 - |x|}{r_2 - r_1}
		\right)
		dx
	\right|
	\le
	\frac{
		C
	}{
		(r_2 - r_1)^m
	}
	\int_{
		B_{r_2}
		\setminus
		B_{r_1}
	}
	|u|
	\,
	dx
$$
and
$$
	\int_{
		B_{r_1}
	}
	g (|u|)
	\,
	dx
	\le
	\int_{
		{\mathbb R}^n
	}
	g (|u|)
	\varphi_0
	\left(
		\frac{r_2 - |x|}{r_2 - r_1}
	\right)
	dx,
$$
we complete the proof.
\end{proof}

\begin{Lemma}\label{L3.2}
Let $u$ be a global weak solution of~\eqref{1.1} and 
$r \le r_1 < r_2 \le 2 r$ 
be positive real numbers.
Then
$$
	J_r (r_2) - J_r (r_1)
	\ge
	C 
	(r_2 - r_1)^m
	g (J_r (r_1)),
$$
where
\begin{equation}
	J_r (\rho)
	=
	\frac{
		1
	}{
		|B_{2 r}|
	}
	\int_{
		B_\rho
	}
	|u|
	\,
	dx.
	\label{L3.2.2}
\end{equation}
\end{Lemma}

\begin{proof}
By Lemma~\ref{L3.1}, we have
$$
	J_r (r_2) - J_r (r_1)
	\ge
	\frac{
		C 
		(r_2 - r_1)^m
	}{
		|B_{2 r}|
	}
	\int_{
		B_{r_1}
	}
	g (|u|)
	\,
	dx
	\ge
	\frac{
		C 
		(r_2 - r_1)^m
	}{
		2^n 
		|B_{r_1}|
	}
	\int_{
		B_{r_1}
	}
	g (|u|)
	\,
	dx.
$$
Thus, to complete the proof it remains to note that
$$
	\frac{
		1
	}{
		|B_{r_1}|
	}
	\int_{
		B_{r_1}
	}
	g (|u|)
	\,
	dx
	\ge
	g 
	\left(
		\frac{
			1
		}{
			|B_{r_1}|
		}
		\int_{
			B_{r_1}
		}
		|u|
		\,
		dx
	\right)
	\ge
	g (J_r (r_1))
$$
since $g$ is a non-decreasing convex function.
\end{proof}

\begin{Lemma}\label{L3.3}
Let $u$ be a global weak solution of~\eqref{1.1} and $r > 0$ be a real number such that
\begin{equation}
		\int_{
			B_r
		}
		|u|
		\,
		dx
		>
		0.
	\label{L3.3.1}
\end{equation}
Then at least one of the following two inequalities is valid:
\begin{equation}
	\int_{
		J_r (r)
	}^{
		J_r (2 r)
	}
	\frac{
		d\zeta
	}{
		g (\zeta / 2)
	}
	\ge
	C r^m,
	\label{L3.3.2}
\end{equation}
\begin{equation}
	\int_{
		J_r (r)
	}^{
		J_r (2 r)
	}
	g^{- 1 / m} (\zeta / 2)
	\zeta^{1 / m - 1}
	\,
	d\zeta
	\ge
	C r,
	\label{L3.3.3}
\end{equation}
where the function $J_r$ is defined by~\eqref{L3.2.2}.
\end{Lemma}

\begin{proof}
Consider a finite sequence of real numbers $\{ r_i \}_{i=0}^l$ constructed as follows.
We take $r_0 = r$. Assume further that $r_i$ is already known.
If $r_i \ge 3 r / 2$, then we put $l = i$ and stop; otherwise we take
$$
	r_{i+1}
	=
	\sup
	\{
		\rho \in [r_i, 2 r] 
		:
		J_r (\rho)
		\le
		2 J_r (r_i)
	\}.
$$
Since $J (r_0) > 0$ and $u \in L_{1, loc} ({\mathbb R}^n)$, this procedure must terminate at a finite step.
In so doing, we obviously have either
\begin{equation}
	r_l = 2 r
	\quad
	\mbox{and}
	\quad
	J_r (r_l) \le 2 J_r (r_{l-1})
	\label{PL3.3.1}
\end{equation}
or
\begin{equation}
	J_r (r_{i+1}) = 2 J_r (r_i),
	\quad
	i = 0, \ldots, l - 1.
	\label{PL3.3.2}
\end{equation}
At first, let~\eqref{PL3.3.1} hold.
By Lemma~\ref{L3.2}, we obtain
$$
	\frac{
		J_r (r_l) - J_r (r_{l-1})
	}{
		g (J_r (r_{l-1}))
	}
	\ge
	C 
	(r_l - r_{l-1})^m.
$$
Since
$$
	\int_{
		J_r (r_{l-1})
	}^{
		J_r (r_l)
	}
	\frac{
		d\zeta
	}{
		g (\zeta / 2)
	}
	\ge
	\frac{
		J_r (r_l) - J_r (r_{l-1})
	}{
		g (J_r (r_{l-1}))
	}
$$
and $r_l - r_{l-1} \ge r / 2$, this yields~\eqref{L3.3.2}.

Now, let~\eqref{PL3.3.2} be valid. 
Lemma~\ref{L3.2} implies that
$$
	\left(
		\frac{
			J_r (r_{i+1}) - J_r (r_i)
		}{
			g (J_r (r_i))
		}
	\right)^{1 / m}
	\ge
	C 
	(r_{i+1} - r_i),
	\quad
	i = 0, \ldots, l - 1.
$$
Combining this with the inequalities
$$
	\int_{
		J_r (r_i)
	}^{
		J_r (r_{i+1})
	}
	g^{- 1 / m} (\zeta / 2)
	\zeta^{1 / m - 1}
	\,
	d\zeta
	\ge
	C
	\left(
		\frac{
			J_r (r_{i+1}) - J_r (r_i)
		}{
			g (J_r (r_i))
		}
	\right)^{1 / m},
	\quad
	i = 0, \ldots, l - 1,
$$
we have
$$
	\int_{
		J_r (r_i)
	}^{
		J_r (r_{i+1})
	}
	g^{- 1 / m} (\zeta / 2)
	\zeta^{1 / m - 1}
	\,
	d\zeta
	\ge
	C 
	(r_{i+1} - r_i),
	\quad
	i = 0, \ldots, l - 1.
$$
Finally, summing the last formula over all $i = 0, \ldots, l - 1$, we conclude that
$$
	\int_{
		J_r (r_0)
	}^{
		J_r (r_l)
	}
	g^{- 1 / m} (\zeta / 2)
	\zeta^{1 / m - 1}
	\,
	d\zeta
	\ge
	C 
	(r_l - r_0).
$$
This implies~\eqref{L3.3.3}.
\end{proof}

We need the following assertion proved in~\cite[Lemma~2.3]{meIzv}.

\begin{Lemma}\label{L3.4}
Let $\psi : (0,\infty) \to (0,\infty)$ and $\gamma : (0,\infty) \to (0,\infty)$
be measurable functions satisfying the condition
$$
	\gamma (\zeta)
	\le
	\essinf_{
		(\zeta / \theta, \theta \zeta)
	}
	\psi
$$
with some real number $\theta > 1$ for almost all $\zeta \in (0, \infty)$.
Also assume that 
$0 < \alpha \le 1$,
$M_1 > 0$,
$M_2 > 0$,
and
$\nu > 1$
are some real numbers with
$M_2 \ge \nu M_1$.
Then
$$
	\left(
		\int_{M_1}^{M_2}
		\gamma^{-\alpha} (\zeta)
		\zeta^{\alpha - 1}
		\,
		d \zeta
	\right)^{1 / \alpha}
	\ge
	A
	\int_{M_1}^{M_2}
	\frac{
		d \zeta
	}{
		\psi (\zeta)
	},
$$
where the constant $A > 0$ depends only on $\alpha$, $\nu$, and $\theta$.
\end{Lemma}

\begin{Lemma}\label{L3.5}
Under the hypotheses of Lemma~$\ref{L3.3}$, let~\eqref{T2.1.1} be valid. Then
\begin{equation}
	\int_{
		J_r (r)
	}^\infty
	g^{- 1 / m} (\zeta / 4)
	\zeta^{1 / m - 1}
	\,
	d\zeta
	\ge
	C r,
	\label{L3.5.1}
\end{equation}
where the function $J_r$ is defined by~\eqref{L3.2.2}.
\end{Lemma}

\begin{proof}
In view of Lemma~\ref{L3.3}, at least one of inequalities~\eqref{L3.3.2}, \eqref{L3.3.3} holds.
In the case where~\eqref{L3.3.3} holds, we obviously have
$$
	\int_{
		J_r (r)
	}^\infty
	g^{- 1 / m} (\zeta / 2)
	\zeta^{1 / m - 1}
	\,
	d\zeta
	\ge
	C r,
$$
whence~\eqref{L3.5.1} follows at once. 
Now, let~\eqref{L3.3.2} be valid.
Lemma~\ref{L3.4} yields
$$
	\left(
		\int_{
			J_r (r)
		}^\infty
		g^{- 1 / m} (\zeta / 4)
		\zeta^{1 / m - 1}
		\,
		d\zeta
	\right)^m
	\ge
	C
	\int_{
		J_r (r)
	}^\infty
	\frac{
		d\zeta
	}{
		g (\zeta / 2)
	}.
$$
Combining this with~\eqref{L3.3.2}, we again obtain~\eqref{L3.5.1}.
\end{proof}

\begin{proof}[Proof of Theorem~$\ref{T2.3}$]
Assume that~\eqref{T2.3.1} holds and, moreover, $u$ is a global weak solution of~\eqref{1.1}.
Since $g$ is a a non-decreasing convex function, we have $g (0) > 0$.
This means that for all $r > 0$ inequality~\eqref{L3.3.1} is valid. 
Really, if $g (0) > 0$, then in accordance with~\eqref{1.2} 
a global weak solution of~\eqref{1.1} can not vanish on a non-empty open set. 
Therefore, in view of Lemma~\ref{L3.5}, for all $r > 0$ estimate~\eqref{L3.5.1} is valid. 
Thus, passing in~\eqref{L3.5.1} to the limit as $r \to \infty$, we arrive at a contradiction.
\end{proof}

\begin{proof}[Proof of Theorem~$\ref{T2.4}$]
Let $r > 0$ be a real number and $u$ be a global weak solution of~\eqref{1.1}.
If $u = 0$ almost everywhere in $B_r$, then~\eqref{T2.4.1} is obvious; 
otherwise~\eqref{L3.3.1} holds and estimate~\eqref{T2.4.1} 
follows from inequality~\eqref{L3.5.1} of Lemma~\ref{L3.5}.
\end{proof}

\begin{proof}[Proof of Theorem~$\ref{T2.1}$]
Let $u$ be a global weak solution of~\eqref{1.1}. 
In view of Theorem~\ref{T2.3}, relation~\eqref{R2.1.1} is valid. 
Thus, we have  $g (0) = 0$ and $G^{-1} (r) \to 0$ as $r \to \infty$. 
In so doing, $G$ is an one-to-one continuous map of the open interval $(0, \infty)$ onto itself and $g$ is an one-to-one continuous map of the closed interval $[0, \infty)$ onto itself.

Lemma~\ref{L3.1} with $r_1 = r / 2$ and $r_2 = r$ yields
\begin{equation}
	\frac{1}{r^m}
	\int_{
		B_r
		\setminus
		B_{r/2}
	}
	|u|
	\,
	dx
	\ge
	C
	\int_{
		B_{r/2}
	}
	g (|u|)
	\,
	dx
	\label{PT2.1.1}
\end{equation}
for all real numbers $r > 0$. 
By Theorem~\ref{T2.4}, this implies the estimate
\begin{equation}
	\int_{
		B_{r/2}
	}
	g (|u|)
	\,
	dx
	\le
	C r^{n - m} G^{-1} (k r)
	\label{PT2.1.2}
\end{equation}
for all real numbers $r > 0$. 
In the case of $n \le m$, passing in~\eqref{PT2.1.2} to the limit as $r \to \infty$,
we obviously obtain $u = 0$ almost everywhere in ${\mathbb R}^n$.
Consequently, we can further assume that $n > m$.

Condition~\eqref{T2.2.1} is equivalent to
$$
	\liminf_{r \to \infty} r^{n - m} G^{-1} (r) < \infty,
$$
whence in accordance with~\eqref{PT2.1.2} it follows that
$$
	\int_{
		{\mathbb R}^n
	}
	g (|u|)
	\,
	dx
	<
	\infty;
$$
therefore,
\begin{equation}
	\int_{
		B_r
		\setminus
		B_{r/2}
	}
	g (|u|)
	\,
	dx
	\to
	0
	\quad
	\mbox{as } r \to \infty.
	\label{PT2.1.3}
\end{equation}
Since $g$ is a convex function, we have
$$
	\frac{
		1
	}{
		\mes (B_r \setminus B_{r/2})
	}
	\int_{
		B_r
		\setminus
		B_{r/2}
	}
	g (|u|)
	\,
	dx
	\ge
	g
	\left(
		\frac{
			1
		}{
			\mes (B_r \setminus B_{r/2})
		}
		\int_{
			B_r
			\setminus
			B_{r/2}
		}
		|u|
		\,
		dx
	\right)
$$
or, in other words,
$$
	\mes (B_r \setminus B_{r/2})
	g^{-1}
	\left(
		\frac{
			1
		}{
			\mes (B_r \setminus B_{r/2})	
		}
		\int_{
			B_r
			\setminus
			B_{r/2}
		}
		g (|u|)
		\,
		dx
	\right)
	\ge
	\int_{
		B_r
		\setminus
		B_{r/2}
	}
	|u|
	\,
	dx
$$
for all real numbers $r > 0$, where $g^{-1}$ is the inverse function to $g$.
By~\eqref{PT2.1.1}, this implies the inequality
$$
	\frac{
		\mes (B_r \setminus B_{r/2})	
	}{
		r^m
	}
	g^{-1}
	\left(
		\frac{
			1
		}{
			\mes (B_r \setminus B_{r/2})	
		}
		\int_{
			B_r
			\setminus
			B_{r/2}
		}
		g (|u|)
		\,
		dx
	\right)
	\ge
	C
	\int_{
		B_{r/2}
	}
	g (|u|)
	\,
	dx
$$
for all real numbers $r > 0$, whence it follows that
\begin{equation}
	\left(
		\int_{
			B_r
			\setminus
			B_{r/2}
		}
		g (|u|)
		\,
		dx
	\right)^{n - m}
	g^{m - n} (f (r)) f^n (r)
	\ge
	C
	\left(
		\int_{
			B_{r/2}
		}
		g (|u|)
		\,
		dx
	\right)^n
	\label{PT2.1.4}
\end{equation}
for all real numbers $r > 0$, where
$$
	f (r)
	=
	g^{-1}
	\left(
		\frac{
			1
		}{
			\mes (B_r \setminus B_{r/2})	
		}
		\int_{
			B_r
			\setminus
			B_{r/2}
		}
		g (|u|)
		\,
		dx
	\right).
$$
Let us note that $f$ is a continuous function and, moreover, $f (r) \to 0$ as $r \to \infty$.
In so doing, since
$$
	G (t)
	\ge
	\int_t^{2 t}
	g^{- 1 / m} (\zeta)
	\zeta^{1 / m - 1}
	\,
	d\zeta
	\ge
	2^{1 / m - 1}
	g^{- 1 / m} (2 t)
	t^{1 / m}
$$
for all $t > 0$, condition~\eqref{T2.2.1} implies the relation
$$
	\liminf_{t \to +0}
	g^{m - n} (t)
	t^n 
	<
	\infty
$$
from which it follows that
\begin{equation}
	\liminf_{r \to \infty}
	g^{m - n} (f(r))
	f^n (r) 
	<
	\infty.
	\label{PT2.1.5}
\end{equation}

Taking into account~\eqref{PT2.1.4}, we obtain
$$
	\liminf_{r \to \infty}
	\left(
		\int_{
			B_r
			\setminus
			B_{r/2}
		}
		g (|u|)
		\,
		dx
	\right)^{n - m}
	g^{m - n} (f (r)) f^n (r)
	\ge
	C
	\left(
		\int_{
			{\mathbb R}^n
		}
		g (|u|)
		\,
		dx
	\right)^n.
$$
In view of~\eqref{PT2.1.3} and~\eqref{PT2.1.5}, 
the limit in the left-hand side of the last expression is equal to zero.
Thus, $u = 0$ almost everywhere in ${\mathbb R}^n$.
\end{proof}


\begin{thebibliography}{100}
\bibitem{Keller}
J.B.~Keller,
On solution of $\Delta u = f(u)$,
Comm. Pure. Appl. Math. 10 (1957) 503--510.

\bibitem{Osserman}
R.~Osserman,
On the inequality $\Delta u \ge f(u)$,
Pacific J. Math. 7 (1957) 1641--1647.

\bibitem{FPR2008}
R. Filippucci, P. Pucci, M. Rigoli,
Non-existence of entire solutions of degenerate elliptic inequalities with weights,
Arch. Ration. Mech. Anal. 188 (2008) 155--179;
Erratum, 188 (2008) 181.

\bibitem{F}
R. Filippucci,
Nonexistence of positive entire weak solutions of elliptic inequalities,
Nonlinear Anal. 70 (2009) 2903--2916.

\bibitem{GR}
M. Ghergu, V. Radulescu,
Existence and nonexistence of entire solutions to the logistic differential equation,
Abstr. and Appl. Anal. 17 (2003) 995--1003.

\bibitem{meNONANA}
A.A.~Kon'kov,
On properties of solutions of quasilinear second-order elliptic inequalities, 
Nonlinear Analysis 123--124 (2015), 89--114. 

\bibitem{MSh}
M. Marcus, A.E. Shishkov,
Fading absorption in non-linear elliptic equations,
Ann. Inst. H.~Poincar\'e Anal. Non Lineaire 30 (2013) 315--336.

\bibitem{NUMZ}
Y. Naito, H. Usami,
Entire solutions of the inequality ${{\rm div} \, (A(|Du|)Du)} = {f (u)}$,
Math. Z. 225 (1997) 167--175.

\bibitem{NUCMB}
Y. Naito, H. Usami,
Nonexistence results of positive entire solutions for quasilinear elliptic inequalities,
Canad. Math. Bull. 40 (1997) 244--253.

\bibitem{SV}
A.E.~Shishkov, L.~Veron,
Admissible initial growth for diffusion equations with weakly superlinear absorption,
Commun. Contemp. Math. 18, 1550089 (2016).

\bibitem{Veron}
L.~Veron,
Comportement asymptotique des solutions d'equations elliptiques
semi-lineaires dans ${\bf  R}^n$,
Ann. Math. Pure Appl. 127 (1981) 25--50.

\bibitem{KS}
A.A.~Kon'kov, A.E.~Shishkov,
On blow-up conditions for solutions of higher order differential inequalities,
Applicable Analysis (to appear).

\bibitem{GMP}
V. A. Galaktionov, E. L. Mitidieri, S. I. Pohozaev, 
Blow-up for higher-order parabolic, hyperbolic, dispersion and Schroedinger equations, 
Monographs and Research Notes in Mathematics, CRC Press, Boca Raton, FL, 2014.


\bibitem{MPbook}
E. Mitidieri, S.I. Pohozaev,
A priori estimates and blow-up of solutions to nonlinear partial
differential equations and inequalities,
Proc. V.A.~Steklov Inst. Math. 234 (2001) 3--383.

\bibitem{meIzv}
A.A.~Kon'kov,
On solutions of non-autonomous ordinary differential equations,
Izv. Math. 65 (2001) 285--327.

\end{thebibliography}
\end{document}